\documentclass[reqno,11pt, a4paper]{article}

\usepackage[utf8]{inputenc}
\usepackage[T1]{fontenc}
\usepackage[english]{babel}
\usepackage[sfdefault,lf]{carlito}

\usepackage{fancyhdr}

\usepackage{hyperref}
\hypersetup{
  colorlinks   = true, 
  urlcolor     = blue, 
  linkcolor    = blue, 
  citecolor   = black 
}

\usepackage[inline]{enumitem}

\usepackage{graphicx}

\usepackage{color}

\usepackage[intlimits]{amsmath}
\usepackage{amssymb, amsfonts}

\usepackage[thmmarks, amsmath, amsthm]{ntheorem}

\usepackage{MnSymbol}
\usepackage{mathbbol}

\usepackage{bbm}

\usepackage{mathrsfs}

\usepackage[all]{xy}

\numberwithin{equation}{section}
\newtheorem{theoremcounter}{theoremcounter}[section]

\theoremstyle{plain}

\newtheorem{corollary}[theoremcounter]{Corollary}
\newtheorem{lemma}[theoremcounter]{Lemma}
\newtheorem{proposition}[theoremcounter]{Proposition}

\newtheorem{theorem}[theoremcounter]{Theorem}

\theoremnumbering{Alph}
\newtheorem{introtheorem}{Theorem}

\theoremstyle{definition}

\newtheorem{definition}[theoremcounter]{Definition}

\theoremstyle{remark}

\newtheorem{notation}[theoremcounter]{Notation}

\theoremstyle{empty}
\theorembodyfont{\normalfont}
\newtheorem{proofseparatehull}{Proof}
\newenvironment{proofseparate}[1]{\begin{proofseparatehull} #1}{\Endofproof \end{proofseparatehull}}

\newcommand{\xqed}[1]{%
    \leavevmode\unskip\penalty9999 \hbox{}\nobreak\hfill
    \quad\hbox{\ensuremath{#1}}}
\newcommand{\Endofproof}{\xqed{\proofSymbol}}

\usepackage{xargs}                      
\usepackage[pdftex,dvipsnames]{xcolor}  
\usepackage[colorinlistoftodos,prependcaption,textsize=tiny]{todonotes}
\newcommandx{\unsure}[2][1=]{\todo[linecolor=red,backgroundcolor=red!25,bordercolor=red,#1]{#2}}
\newcommandx{\change}[2][1=]{\todo[linecolor=blue,backgroundcolor=blue!25,bordercolor=blue,#1]{#2}}
\newcommandx{\info}[2][1=]{\todo[linecolor=OliveGreen,backgroundcolor=OliveGreen!25,bordercolor=OliveGreen,#1]{#2}}
\newcommandx{\improvement}[2][1=]{\todo[linecolor=Plum,backgroundcolor=Plum!25,bordercolor=Plum,#1]{#2}}

\newcommand{\cU}{\ensuremath{\mathcal{U}}}

\newcommand{\cZ}{\ensuremath{\mathcal{Z}}}

\newcommand{\bT}{\ensuremath{\mathbb{T}}}

\newcommand{\rE}{\ensuremath{\mathrm{E}}}

\newcommand{\rH}{\ensuremath{\mathrm{H}}}

\newcommand{\rK}{\ensuremath{\mathrm{K}}}
\newcommand{\rL}{\ensuremath{\mathrm{L}}}

\newcommand{\rS}{\ensuremath{\mathrm{S}}}

\newcommand{\rZ}{\ensuremath{\mathrm{Z}}}

\newcommand{\veps}{\ensuremath{\varepsilon}}

\newcommand{\vphi}{\ensuremath{\varphi}}

\newcommand{\ol}{\overline}

\newcommand{\eqstop}{\ensuremath{\, \text{.}}}
\newcommand{\eqcomma}{\ensuremath{\, \text{,}}}

\newcommand{\NN}{\ensuremath{\mathbb{N}}}
\newcommand{\ZZ}{\ensuremath{\mathbb{Z}}}
\newcommand{\QQ}{\ensuremath{\mathbb{Q}}}
\newcommand{\RR}{\ensuremath{\mathbb{R}}}
\newcommand{\CC}{\ensuremath{\mathbb{C}}}

\newcommand{\Hom}{\ensuremath{\mathop{\mathrm{Hom}}}}
\newcommand{\id}{\ensuremath{\mathrm{id}}}
\newcommand{\ra}{\ensuremath{\rightarrow}}

\newcommand{\hra}{\ensuremath{\hookrightarrow}}
\newcommand{\thra}{\ensuremath{\twoheadrightarrow}}

\newcommand{\ev}{\ensuremath{\mathrm{ev}}}

\newcommand{\Aut}{\ensuremath{\mathrm{Aut}}}

\newcommand{\ot}{\ensuremath{\otimes}}

\newcommand{\Cstar}{\ensuremath{\mathrm{C}^*}}

\newcommand{\Wstar}{\ensuremath{\mathrm{W}^*}}

\newcommand{\bo}{\ensuremath{\mathcal{B}}}

\newcommand{\im}{\ensuremath{\mathop{\mathrm{im}}}}

\newcommand{\Cstarred}{\ensuremath{\Cstar_\mathrm{red}}}

\newcommand{\cont}{\ensuremath{\mathrm{C}}}

\newcommand{\conto}{\ensuremath{\mathrm{C}_0}}

\newcommand{\Linfty}{\ensuremath{{\offinterlineskip \mathrm{L} \hskip -0.3ex ^\infty}}}

\newcommand{\ltwo}{\ensuremath{\ell^2}}

\newcommand{\Out}{\ensuremath{\mathrm{Out}}}

\newcommand{\grpaction}[1]{\ensuremath{\stackrel{#1}{\curvearrowright}}}

\newcommand{\Prim}{\operatorname{Prim}}
\newcommand{\Glimm}{\operatorname{Glimm}}
\newcommand{\spec}{\operatorname{spec}}
\newcommand{\Tone}{\ensuremath{\mathrm{T}_1}}
\newcommand{\rank}{\operatorname{rank}}

\newcommand{\authors}{Caleb Eckhardt and Sven Raum}
\renewcommand{\title}{C*-superrigidity of 2-step nilpotent groups}
\newcommand{\shorttitle}{C*-superrigidity of 2-step nilpotent groups}

\begin{document}


\thispagestyle{empty}

\begin{center}
  \begin{minipage}[c]{0.9\linewidth}
    \textbf{\LARGE \title} \\[0.5em]
    by \authors
  \end{minipage}
\end{center}
  
\vspace{1em}

\renewcommand{\thefootnote}{}
\footnotetext{
  last modified on \today
}

\begin{center}
  \begin{minipage}{0.8\linewidth}
    \textbf{Abstract}.
    We show that torsion-free finitely generated nilpotent groups are characterised by their group C*-algebras and we additionally recover their nilpotency class as well as the subquotients of the upper central series.  We then use a C*-bundle decomposition and apply K-theoretic methods based on noncommutative tori to prove that every torsion-free finitely generated 2-step nilpotent group can be recovered from its group C*-algebra.
  \end{minipage}
\end{center}


\section{Introduction}
\label{sec:introduction}

It is a classical problem to recover a discrete group $G$ from its various group algebras such as the integral group ring $\ZZ G$ and the complex group ring $\CC G$.  Notably the Kadison-Kaplansky unit conjecture predicts that for a torsion-free group $G$ every unit of $\CC G$ is of the form $c u_g$, for $c \in \CC^\times$ and $g \in G$, making it possible to recover $G$ canonically from $\CC G$.  In the operator algebraic setting, studying completions of the *-algebra $\CC G$ in different topologies of the bounded operators $\bo(\ltwo G)$, such as the reduced group \Cstar-algebra $\Cstarred(G)$ or the group von Neumann algebra $\rL(G)$, similar questions have attracted strong interest in the past years.

In the 80's, Connes conjectured in \cite[page 45]{connes82} that in analogy with Mostow and Margulis superrigidity every discrete group with infinite conjugacy classes and Kazhdan's property (T) can be recovered from its group von Neumann algebra; that is if $G$ is such a group and $H$ is any other group with $\rL(G)\cong \rL(H)$ then $G\cong H$ .  Derived from an old name for von Neumann algebras, this phenomenon is termed \mbox{\Wstar-superrigidity}.  In contrast to Connes' early vision of how the subject might develop, only in \cite{ioanapopavaes10} the first breakthrough on \mbox{\Wstar-superrigidity} was achieved, showing that certain iterated wreath products are \mbox{\Wstar-superrigid}.  In the introduction to their article, Ioana, Popa and Vaes draw attention to the even more mysterious situation in \mbox{\Cstar-superrigidity} and in analogy with Kadison-Kaplansky's unit conjecture point out that: ``It seems not even known whether $\Cstarred(G)$ always remembers a torsion-free group $G$''-- in other words, if $G$ is torsion-free group and  $\Cstarred(G)\cong \Cstarred(H)$ for some other group $H$, does it follow that $G \cong H$?  At the time of writing their article, only torsion-free abelian groups $G$ were known to be \mbox{\Cstar-superrigid} by a classical result of Scheinberg \cite[Theorem 1]{scheinberg74}.  More precisely, $G$ is isomorphic to the unitary group of $\Cstarred(G)$ modulo the connected component of the identity \cite[Theorem 8.58]{hofmannmorris98}.  Note that the restriction to torsion-free groups in Ioana-Popa-Vaes' remark is necessary and natural, since for example $\Cstarred(\ZZ_2 \times \ZZ_2)\cong \Cstarred(\ZZ_4)$ and the group \Cstar-algebra of every infinite abelian torsion group is isomorphic to the continuous functions on the standard Cantor set.

Motivated by \Cstar-superrigidity of torsion-free abelian groups, certain torsion-free virtually abelian groups were shown to be \Cstar-superrigid in \cite{kundbyraumthielwhite17}, providing the first examples of non-abelian torsion-free \Cstar-superrigid groups.  Further \Cstar-superrigidity results were deduced in \cite{chifanioana17} by combining the authors' result on  \Wstar-superrigidity of amalgamated free products with the unique trace property \cite{breuillardkalantarkennedyozawa14} of their reduced group \Cstar-algebra.  In the present article we prove the following result, providing the first natural class of \Cstar-superrigid non-abelian torsion-free groups.
\begin{introtheorem}
  \label{introthm:superrigid}
  Every torsion-free finitely generated 2-step nilpotent group is \Cstar-superrigid.
\end{introtheorem}
We mention that in analogy with abelian groups, a torsion-free non-abelian nilpotent group $G$ is as far from being \Wstar-superrigid as possible.  Indeed,  it follows from Connes's \cite{connes76}  and the theory of direct integrals that $\rL(G)\cong \Linfty([0,1]) \otimes R$, where $R$ denotes the hyperfinite ${\rm II}_1$ factor. 

Our global strategy to prove \Cstar-superrigidity follows the same lines as \cite{kundbyraumthielwhite17}: we first characterise torsion-free finitely generated nilpotent groups in terms of their group \Cstar-algebras and recover their nilpotency class.  This is the second main result of this article.
\begin{introtheorem}
  \label{introthm:nilpotent}
  Let $G$ be a torsion-free finitely generated nilpotent group and $H$ some discrete group.  If $\Cstarred(H) \cong \Cstarred(G)$, then $H$ is a torsion-free finitely generated nilpotent group, too. Further the class of $H$ equals the class of $G$ and the subquotients of the upper central series of both groups agree.
\end{introtheorem}

\noindent An important ingredient in the proof of this theorem describes the centre of the group \Cstar-algebra of an arbitrary torsion-free virtually nilpotent group, which is of independent interest.
Given a torsion-free virtually nilpotent group $G$ with some maximal normal finite index nilpotent subgroup $N \unlhd G$ we write $F = G/N$ for the finite quotient group.  Then the outer action of $F$ on $N$ induces an action of $F$ on the centre $\cZ(N)$, as is explained in Section \ref{sec:centre}.
\begin{introtheorem}
  \label{introthm:centre}
  If $G$ is a torsion-free virtually nilpotent group and $N \unlhd G$ is a maximal normal finite index nilpotent subgroup with quotient $G/N = F$, then $\cZ(\Cstar G) = \Cstar(\cZ(N))^F$.
\end{introtheorem}

\noindent Let us now explain the strategy to recover a torsion-free finitely generated 2-step nilpotent group $G$ from its group \Cstar-algebra.  A first application of Theorem \ref{introthm:nilpotent} makes it clear that any group $H$ satisfying $\Cstarred(G) \cong \Cstarred(H)$ must be torsion-free finitely generated and nilpotent of class 2 and satisfy $\cZ(H) \cong \cZ(G)$.  Applying Theorem \ref{introthm:centre} to write $\Cstarred(H) \cong \Cstarred(G)$ as a \Cstar-bundle over their centre, we are able to recover an isomorphism of torsion-free abelian groups $H/\cZ(H) \cong G/\cZ(G)$.  Further, analysing how K-theory varies over the different fibres of this \Cstar-bundle, we recover the 2-cocycle describing the central extension $1 \ra \cZ(G) \ra G \ra G / \cZ(G) \ra 1$.  In fact, the fibres of $\Cstarred(G) \cong \Cstarred(H)$ are so called noncommutative tori, whose K-theory has been intensively studied by Elliott \cite{elliott84} and Rieffel \cite{rieffel88} and we make use of the former results, summarised in Section \ref{sec:elliott}.

\noindent It is natural to ask for possible extensions of the present result, either beyond the scope of finitely generated groups or beyond the class of 2-step nilpotent groups.  As for the first extension, we are facing the problem to find an appropriate replacement for Theorem \ref{thm:quotient-implies-free}, which plays an important role in Section \ref{sec:characterising}.  Note that both theorems in question make crucial use of finite rank of the abelian groups involved.  As for the second extension, while the general strategy to prove C*-superrigidity still applies to finitely generated, torsion-free nilpotent groups, important challenges have to be overcome.  Understanding of K-theory and the trace pairing of twisted group \Cstar-algebras of nilpotent groups as well as the replacement for cup products as used in the proof of Theorem \ref{thm:recovering-triple} are the two main such challenges.  Nevertheless, after this work has first appeared in form of a preprint some examples of \Cstar-superrigid nilpotent groups of arbitrary class have been provided \cite{omland18-superrigid}.

This article has 5 sections.  After the introduction, thorough preliminaries provide all necessary information to make this article readable by non-experts in operator algebras, group theory or cohomology theory.  In Section \ref{sec:centre} we calculate the centre of the group \Cstar-algebra of a virtually nilpotent group, proving Theorem \ref{introthm:centre}.  In Section \ref{sec:characterising} we provide the characterisation of torsion-free finitely generated nilpotent groups announced in Theorem \ref{introthm:nilpotent}.  We then proceed to prove our main Theorem \ref{introthm:superrigid} in Section \ref{sec:superrigidity}.

\subsection*{Acknowledgements}
The authors would like to thank Elizabeth Gillaspy for informing us about each others interest in the present article's subject.  Further, we thank R{\'e}mi Boutonnet for useful remarks on a previous version of this work and we thank Alain Valette for providing us with the idea of an elementary proof for torsion-free nilpotent groups of Theorem \ref{thm:chacterising-tf} as well as pointing out the references \cite{jipedersen90}, \cite{kaniuthtaylor89} and \cite{rosenberg83}.  We thank Yuhei Suzuki for pointing out an error in \cite[Proposition 3.3]{blanchardgogic15}, which we used in our work, and providing us with a shorter proof of Theorem \ref{thm:quotient-implies-free}.  Finally, we would like to thank the referee for his careful reading of the manuscript and pointing out several inaccuracies.  C.E. was partially supported by a grant from the Simons Foundation.

\section{Preliminaries and various results}
\label{sec:preliminaries}
All groups considered in this article are discrete unless explicitly stated differently.  A group $G$ is called nilpotent if its upper central series defined by
\begin{equation*}
  Z_0 = \{e\}
  \eqcomma \quad
  Z_1 = \cZ(G)
  \quad \text{ and } \quad
  Z_{i+1} = \{g \in G \mid \forall h \in G: [g,h] \in Z_i\}
\end{equation*}
terminates in some $Z_n = G$.  The minimal $n$ such that $Z_{n+1} = G$ is called the nilpotency class of $G$, and we say that $G$ is $n$-step nilpotent.

We refer the reader to the excellent exposition \cite{brownozawa08} on discrete groups and their operator algebras.  Since nilpotent groups are amenable, their maximal \Cstar-algebra $\Cstar(G)$ and their reduced \Cstar-algebra $\Cstarred(G)$ agree.  The latter is defined as the closure of the complex group ring $\CC G$ under the image of the left-regular representation $\lambda: G \ra \cU(\ltwo G): \lambda_g \delta_h = \delta_{gh}$.  For amenable groups, we write $\Cstar(G)$ for this group \Cstar-algebra.


Let us recall the notion of conditional expectations between \Cstar-algebras.
If $B \subset A$ is a unital \Cstar-subalgebra, then a conditional expectation from $A$ onto $B$ is a contractive projection $\rE: A \ra B$ of norm one.   By Tomiyama's theorem \cite[Theorem 1.5.10]{brownozawa08}, every conditional expectation $\rE: A \ra B$ is $B$-bi-modular, that is $\rE(b a b') = b \rE(a) b'$ for all $b, b' \in B$ and all $a \in A$.  If $H \leq G$ is an inclusion of groups, then there is a natural faithful conditional expectation $\rE: \Cstarred(G) \ra \Cstarred(H)$ uniquely defined by the property $\rE(u_g) = \mathbb{1}_H(g) u_g$.


\subsection{Actions on torsion-free abelian groups}
\label{sec:torsion-free-abelian}

In this section we prove Theorem \ref{thm:quotient-implies-free}, which might be known to experts.  In order to remain consistent with the notation of the remaining article, we write abelian groups multiplicatively.

Recall that an action of a discrete group $G$ on a topological space $X$ is called topologically free, if for every $g \in G \setminus \{e\}$ the fixed point set $\{x \in X \mid gx = x\}$ is meager in $X$.  The next lemma is folklore and a proof can be found for example in \cite[Lemma 2.1]{eckhardt15}.
\begin{lemma}
  \label{lem:finite-group-acting-on-zn}
  Let $H$ be a connected topological group and let $G$ be a group acting faithfully by continuous automorphisms on $G$.  Then $G \grpaction{} H$ is topologically free.
\end{lemma}

\noindent The next theorem's proof makes use of the notion of rank.  If $A$ is an abelian group considered as a $\ZZ$-module, the rank of $A$ is defined as $\rank(A) = \dim_\QQ A \ot_\ZZ \QQ$.  Further, if $A$ is an abelian group, then the Pontryagin dual $\hat A = \Hom(A, \rS^1)$ equipped with the topology of pointwise convergence is a compact abelian group.  Every automorphism $\alpha$ of $A$ defines a continuous automorphism of $\hat A$ by precomposition $\chi \mapsto \chi \circ \alpha$.  Finally, if a group $G$ acts on a \Cstar-algebra $A$ by *-automorphisms, we denote by $A^G = \{x \in A \mid gx = x \textup{ for all }g\in G\}$ the fixed point algebra.  We thank Yuhei Suzuki for indicating a shorter proof for the next theorem than the one originally presented.   We start by a short lemma.
\begin{lemma}
  \label{lem:finite-group-action-trivial}
  Let $F \grpaction{} \bT^n$ be a finite group acting by continuous automorphisms such that the quotient space satisfies $\bT^n/F \cong \bT^n$.  Then $F$ acts trivially.
\end{lemma}
\begin{proof}
  The quotient map $q: \bT^n \ra \bT^n/F$ by a finite group has the path lifting property according to \cite[11.1.4]{brown06-topology-groupoids}.  So the induced map on fundamental groups $\pi_1(q)$ is a surjection, hence an isomorphism because $\ZZ^n \cong \pi_1(\bT^n)$ is Hopfian.  Identifying $\ZZ^n \cong \pi_1(\bT^n)$ by the map that sends $a \in \ZZ^n$ to the image of the line segment $[0,a] \subset \RR^n$, we observe that $\pi_1(q)$ factors through the group $\ZZ^n/\langle a - f a \mid a \in \ZZ^n, f \in F \rangle$.  Since every proper quotient of $\ZZ^n$ has rank strictly less than $n$ and $\pi_1(q)$ is an isomorphism, it follows that $a = fa$ for all $a \in \ZZ^n$ and $f \in F$.  This finishes the proof of the theorem.  
\end{proof}

\begin{theorem}
  \label{thm:quotient-implies-free}
  Let $A$ be a torsion-free abelian group and $G$ a group acting on $A$ by automorphisms with finite orbits.  If $\cont(\hat A)^G \cong \cont(\bT^n)$ for some $n \in \NN$, then $A \cong \ZZ^n$ and $G$ acts trivially on $A$.
\end{theorem}
\begin{proof}
  Since $G$ acts with finite orbits on $A$, we can write $A = \bigcup_{i \in I} A_i$ as a directed union of subgroups, each generated by a finite $G$-invariant set.  Since every automorphism of $A_i$, which is trivial on its finite generating set, is trivial on $A_i$, it follows that $G \grpaction{} A_i$ factors through a finite group, say $G \thra F_i$.  Note that since $A$ is torsion free, $\hat A_i$ is connected for each $i$, so that Lemma \ref{lem:finite-group-acting-on-zn} says that $F_i$ acts topologically freely on $\hat A_i$.  So $\hat A_i/F_i$ contains an open and dense subset that is locally Euclidean of dimension $\rank A_i$.  In particular the covering dimension of $\hat A_i/F_i$ equals the rank of $A_i$.  (See e.g. \cite[p.305]{munkres00} for the notion of covering dimension).

  We put $X = \spec{\Cstar(A)^G}$ and observe that the inclusions $\Cstar(A)^G \supset \Cstar(A_i)^{F_i}$ and $\Cstar(A) \supset \Cstar(A)^G$ provide us with quotient maps $X \thra \hat A_i/F_i$ and $\hat A \thra X$.  The following commutative diagram shows that the former is an open map.
  \begin{equation*}
    \xymatrix
    {
      X \ar@{->>}[r] & \hat A_i/F_i \\
      \hat A  \ar@{->>}[u]^{\text{cont}} \ar@{->>}[r]_{\text{open}} & \hat A_i \ar@{->>}[u]_{\text{open}}
    }
  \end{equation*}
  Indeed, both maps marked as open in the diagram arise as quotient maps from group actions.  By hypothesis, we have $X \cong \bT^n$.   It follows that
  \begin{equation*}
    \rank(A_i) = \dim \hat A_i/F_i \leq \dim X = n
    \eqstop
  \end{equation*}
  We showed that the rank of $(A_i)_i$ is uniformly bounded by $n$, so that $\rank A \leq n$ follows.  So there is a finite $G$-invariant subset $S \subset A$ such that writing $B = \langle S \rangle$, the division closure satisfies $A = \{g \in A \mid \exists k \in \NN_{\geq 1}: g^k \in B\}$.  The action of $G$ on $B$ factors through a finite group $F$, and so does the action on $A$: indeed, if $a \in A$ and $k \in \NN_{\geq 1}$ such that $a^k \in B$, then $(g a)^k=g a^k = a^k$ implies that $((ga) \cdot a^{-1})^k = e$.  Since $A$ is torsion-free, it follows that $g a = a$.

  To finish the proof note that $F \grpaction{} \hat A$ is topologically free by Lemma \ref{lem:finite-group-acting-on-zn}, since $\hat A$ is connected.  Further, $\cont(\hat A/F) \cong \Cstar(A)^G \cong \cont(\bT^n)$ showing that $\hat A$ is locally euclidean of dimension $n$, which in turn implies that $\hat A \cong \bT^n$ by the structure theorem for abelian groups \cite[Theorem~4.2.4]{deitmarechterhoff14}.  So $A \cong \ZZ^n$ and by Lemma \ref{lem:finite-group-action-trivial}, we find that $F$ acts trivially on $A$.  So also $G$ acts trivially on $A$.
\end{proof}

\subsection{ FC-groups and the centre of group C*-algebras}
\label{sec:fc-groups}

In this section, we give a short account of Proposition \ref{prop:torsion-free-fc}, providing a description of the centre of a reduced group \Cstar-algebra by means of certain abelian subgroups.  To this end we will analyse its FC-centre.  An FC-group is a group $G$  whose conjugacy classes are finite.  More generally, the FC-centre of a group $G$ is $\{g \in G \mid g \text{ has a finite conjugacy class}\}$.
The following theorem describes finitely generated FC-groups.
\begin{theorem}[\mbox{Special case of Theorem 2 of \cite{duguidmclain55}}]
  \label{thm:fg-fc-hypercentral}
  A finitely generated FC-group is virtually abelian.
\end{theorem}

\noindent The following proposition is classical.  We provide a short proof for the convenience of the reader.
\begin{proposition}
  \label{prop:torsion-free-fc}
  A torsion-free FC-group is abelian.
\end{proposition}
\begin{proof}
  It suffices to prove that every finitely generated torsion-free FC-group $G$ is abelian.  By Theorem \ref{thm:fg-fc-hypercentral} we know that $G$ is virtually abelian.  So \cite[Theorem 2]{bear48} implies that $\cZ(G) \leq G$ has finite index.  It follows that the commutator subgroup of $G$ is finite, and hence trivial since $G$ is torsion-free.
\end{proof}

\noindent The next proposition describes the centre of the group \Cstar-algebra of a torsion-free group in terms of an abelian group.
\begin{proposition}
  \label{prop:centre-general-group-cstar-algebra}
  Let $G$ be a torsion-free group.   Then the FC-centre of $G$ is a normal abelian subgroup $A \unlhd G$ satisfying $\cZ(\Cstarred(G)) \subset \Cstar(A)$.  In particular, $\cZ(\Cstarred(G)) =  \Cstar(A)^G$.
\end{proposition}
\begin{proof}
  The FC-centre $A \unlhd G$ is a torsion-free FC-group and hence abelian by Proposition \ref{prop:torsion-free-fc}.  We have to show that $\cZ(\Cstarred G) \subset \Cstar(A)$.  If $x \in \cZ(\Cstarred G)$ we write $\hat x = x \delta_e = \sum_{g \in G} x_g \delta_g \in \ltwo(G)$.  Since $x$ is central, we have $\hat x = \widehat{u_h x u_h^*}$ for all $h \in G$ and hence $x_g = x_{hgh^{-1}}$ for all $h \in G$.  Since $(x_g)_g$ is $\ell^2$-summable, it follows that $x_g = 0$, unless $g$ has a finite conjugacy class in $G$, meaning $g \in A$.  So $x \in \Cstar(A)$.
\end{proof}

\subsection{Central extensions and cocycles}
\label{sec:central-extensions}

In this section we clarify the exact relation between $n$-step nilpotent groups and extension data for $n-1$-step nilpotent groups.  The book \cite{brown10-cohomology} provides a reference for the cohomology of discrete groups.  Recall that given a group $H$ and an abelian group $A$, there is a one-to-one correspondence between central extensions
\begin{equation*}
  \xymatrix{
    1 \ar[r] & A \ar[r] & G \ar[r] & H \ar[r] & 1
  }
\end{equation*}
and elements of $\rH^2(H, A)$, which sends an extension after choice of a section $s: H \ra G$ to the 2-cocycle $\sigma(h_1, h_2) = s(h_1h_2)^{-1} s(h_1)s(h_2)$, which we call an extension cocycle  \cite[Theorem 3.12]{brown10-cohomology}.  Two central extensions giving rise to $G_1$, $G_2$ are called equivalent, if there is an isomorphism $G_1 \ra G_2$ making the following diagram commute
\begin{equation*}
  \xymatrix{
    1 \ar[r] & A \ar[r]  \ar^{\id}[d] & G_1 \ar[r] \ar[d] & H \ar[r] \ar^{\id}[d] & 1 \\
    1 \ar[r] & A \ar[r] & G_2 \ar[r] & H \ar[r] & 1 \eqstop
  }
\end{equation*}
Note that the automorphisms of $A$ and $H$ are fixed to be the identity.  Aiming at a classification of central extensions without this restriction, we introduce the action of $\Aut(H) \times \Aut(A)$ on $\rH^2(H, A)$ that is induced by $\rho \mapsto \psi_A \circ \rho \circ (\psi_H \times \psi_H)$ with $\psi_H \in \Aut(H)$, $\psi_A \in \Aut(A)$, $\rho \in \rZ^2(H, A)$.

Before we proceed to spell out the relation between $n$-step nilpotent groups and extension data for $n-1$-step nilpotent groups, we need to introduce the following notation.
\begin{notation}
  \label{not:centrally-non-degenerate}
  Let $H$ be a group and $A$ an abelian group.  A 2-cocycle $\sigma \in \rZ^2(H, A)$ is called \emph{centrally non-degenerate} if for all non-trivial elements $g \in\cZ(H)$ there is $h \in H$ such that $\sigma(g,h) \neq \sigma(h,g)$.
\end{notation}
The following theorem must be well-known.  Not being able to find a reference we indicate a proof for the reader's convenience.
\begin{theorem}
  \label{thm:classify-central-extensions}
  Let $H$ be a group and $A$ an abelian group. Associating to a central extension of $A$ by $H$ the extension 2-cocycle in $\rH^2(H,A)$, we obtain a bijection between
  \begin{itemize}
  \item isomorphism classes of groups $G$ such that $\cZ(G) \cong A$ and $G/\cZ(G) \cong H$, and
  \item $\Aut(H) \times \Aut(A)$ orbits of centrally non-degenerate cocycles in $\rH^2(H, A)$.
  \end{itemize}
\end{theorem}
\begin{proof}
  We provide maps which are inverse to each other.  Given $G$ such that $\cZ(G) \cong A$ and $G/\cZ(G) \cong H$, we choose isomorphisms $\vphi_A: \cZ(G) \ra A$ and $\vphi_H: G/ \cZ(G) \ra H$.  This provides us with an extension
  \begin{equation*}
    \xymatrix
    {
      1 \ar[r] & A \ar[r] & G \ar[r] & H \ar[r] & 1
    }
  \end{equation*}
  and after choice of a section $s: H \ra G$ with a cocycle $\sigma \in \rZ^2(H, A)$, whose cohomology class is independent of the choice of the section.  Note that $\sigma$ is centrally non-degenerate, since $s(H) \cap \cZ(G) = \{e\}$.  Further, the orbit of $[\sigma] \in \rH^2(H, A)$ for the $\Aut(H) \times \Aut(A)$ action is independent of the choice of $\vphi_A$ and $\vphi_H$.  So we obtain a well-defined map assigning to the group $G$ an $\Aut(H) \times \Aut(A)$ orbit of centrally non-degenerate cocycles in $\rH^2(H, A)$.  Since every isomorphism $G \cong \tilde G$ restricts to an isomorphism of $\cZ(G) \cong \cZ(\tilde G)$ and induces an isomorphism $G/\cZ(G) \cong \tilde G / \cZ(\tilde G)$, this assignment descends to a map from isomorphism classes of groups $G$ satisfying that $\cZ(G) \cong A$ and $G/\cZ(G) \cong H$ to $\Aut(H) \times \Aut(A)$ orbits of centrally non-degenerate cocycles in $\rH^2(H, A)$.

  Let now $\sigma \in \rZ^2(H, A)$ and define $G = A \rtimes_\sigma H$, which as a set is the Cartesian product $A \times H$ and becomes a group with the multiplication
  \begin{equation*}
    (a_1, h_1) (a_2, h_2) = (a_1 a_2 \sigma(h_1, h_2), h_1h_2)
    \eqstop
  \end{equation*}
  The isomorphism class of the extension
  \begin{equation*}
    \xymatrix@R=0pc
    {
      1 \ar[r] & A \ar[r] & G \ar[r] & H \ar[r] & 1 \\
      & a \ar@{|->}[r] & (a, e) &&
    }
  \end{equation*}
  does only depend on the class of $\sigma$ in $\rH^2(H, A)$.  Note that if $\sigma$ is centrally non-degenerate, then $\cZ(G) = A$.  Further, if $\psi_A \in \Aut(A)$, $\psi_H \in \Aut(H)$ and we define $\tilde \sigma = \psi_A \circ \sigma \circ (\psi_H^{-1} \times \psi_H^{-1})$ as well as $\tilde G = A \rtimes_{\tilde \sigma} H$, then
  \begin{equation*}
    G \ni (a, h) \mapsto (\psi_A(a), \psi_H(h)) \in \tilde G
  \end{equation*}
  is a group isomorphism.  So we obtain a well-defined map from $\Aut(H) \times \Aut(A)$ orbits in $\rH^2(H, A)$ to isomorphism classes of groups $G$ such that $\cZ(G) \cong A$ and $G/\cZ(G) \cong H$.  
\end{proof}

\noindent We specialise the statement of Theorem \ref{thm:classify-central-extensions} to the case of nilpotent groups.  For later use in Section~\ref{sec:superrigidity}, the next corollary will be reformulated in Section \ref{sec:cocycles-skew-symmetric-forms}.
\begin{corollary}
  \label{cor:extension-data}
  Associating to a nilpotent group $G$ the triple $(\cZ(G), G/\cZ(G),\sigma_G)$, where we denote by $\sigma_G \in \rH^2(G/\cZ(G), \cZ(G))$ the extension cocycle, we obtain a bijection between
  \begin{itemize}
  \item isomorphism classes of $n$-step nilpotent groups, and
  \item equivalence classes of triples $(A, H, \sigma)$ with $A$ abelian, $H$ an $n-1$-step nilpotent group and $\sigma \in \rH^2(H, A)$ a centrally non-degenerate cocycle.  Two triples $(A, H, \sigma)$ and $(\tilde A, \tilde H, \tilde \sigma)$ are equivalent if there are isomorphisms  $\vphi_A: A \ra \tilde A$ and $\vphi_H: H \ra \tilde H$ such that $\tilde \sigma \circ (\vphi_H \times \vphi_H) = \vphi_A \circ \sigma$.
  \end{itemize}
\end{corollary}

\subsection{Cocycles and skew-symmetric forms}
\label{sec:cocycles-skew-symmetric-forms}

In this section we recall the relationship between 2-cocycles on torsion-free abelian groups and skew-symmetric forms.  It provides the reformulation of Corollary \ref{cor:extension-data} as it will be used in Section \ref{sec:superrigidity}.  The references \cite[Theorem 7.1]{kleppner63}, \cite[Chapter V.6 Exercise 5]{brown10-cohomology} or \cite[Proof of Theorem 2]{hughes51} show this result without the statement of $\Aut(B) \times \Aut(A)$-equivariance, which can be right away checked by the formula provided below.  We denote by $B \wedge B$ the wedge product, which arises as the quotient group $B \otimes_\ZZ B/ \langle b \otimes b' + b' \otimes b \mid b, b' \in B \rangle$.
\begin{theorem}
  \label{thm:cocycles-skew-symmetric-forms}
  Let $B$ be a free abelian group and $A$ any abelian group.  Then there is an $\Aut(B) \times \Aut(A)$-equivariant bijection between $\rH^2(B, A)$ and $A$-valued skew-symmetric forms on $B$, given by
  \begin{equation*}
    \rH^2(B, A) \ra \Hom(B \wedge B, A):
    \sigma \mapsto (b_1 \wedge b_2 \mapsto \sigma(b_1, b_2) - \sigma(b_2, b_1))
    \eqstop
  \end{equation*}
\end{theorem}

\begin{corollary}
  \label{cor:extension-data-skew-symmetric}
  Associating to a 2-step nilpotent group $G$ the triple $(\cZ(G), G/\cZ(G), \omega_G)$, where we denote by $\omega_G \in \Hom(G/ \cZ(G) \wedge G/ \cZ(G) , \cZ(G))$ the skew-symmetrisation of the extension cocycle, we obtain a bijection between
  \begin{itemize}
  \item isomorphism classes of torsion-free finitely generated 2-step nilpotent groups, and
  \item equivalence classes of triples $(A, B, \omega)$ with $A,B$ torsion-free finitely generated abelian and $\omega \in \Hom(B \wedge B, A)$ a $A$-valued skew-symmetric form.  Two triples $(A, B, \omega)$ and $(\tilde A, \tilde B, \tilde \omega)$ are equivalent if there are isomorphisms  $\vphi_A: A \ra \tilde A$ and $\vphi_B: B \ra \tilde B$ such that $\tilde \omega \circ (\vphi_B \wedge \vphi_B) = \vphi_A \circ \omega$.
  \end{itemize}
\end{corollary}

\subsection{Homotopy classes of evaluation maps}
\label{sec:characters}

The following statement is easily checked on generators and will be used in Section \ref{sec:superrigidity}.
\begin{proposition}
  \label{prop:pairing-zn}
  Let $h: \rS^1 \ra \bT^n$ be a loop, let $g \in \ZZ^n \cong \hat \bT^n$ and let $\ev_g: \bT^n \ra \rS^1$ be the evaluation map.  The natural identification $\pi_1( \bT^n) \cong \ZZ^n$ induces a bilinear form $\langle \phantom{x}, \phantom{x} \rangle$ such that $\pi_1(\ev_g)([h]) = \langle g , [h] \rangle$ holds under the natural identification $\pi_1(\rS^1) \cong \ZZ$.
\end{proposition}
\begin{proof}
  Note that both sides of the equation $\pi_1(\ev_g)([h]) = \langle g , [h] \rangle$ only depend on the homotopy class of $h$.  The identification $\pi_1( \bT^n) \cong \ZZ^n$ provides us with the standard loop $h(e^{2\pi i t}) = (e^{2 \pi i h_1 t}, \dotsc, e^{2 \pi i h_n t})$, for $(h_1, \dotsc, h_n) \in \ZZ^n$.  Further $g = (g_1, \dotsc, g_n) \in \ZZ^n$ defines the character $\ev_g(e^{2 \pi i t_1}, \dotsc, e^{2 \pi i t_n}) = e^{2 \pi i \sum_j g_j t_j}$.  It follows that $(\ev_g \circ h)(e^{2 \pi i t}) = e^{2 \pi i \sum_j h_j g_j}$ and hence $\pi_1(\ev_g)([h]) = [\ev_g \circ h] = \sum_j h_j g_j = \langle g, [h] \rangle$ under the natural identification $\pi_1(\rS^1) \cong \ZZ$.
\end{proof}

\subsection{$\cont(X)$-algebras}
\label{sec:contoX-algebras}

In this section we recall the notion of $\cont(X)$-algebras, which is for example explained in \cite[Appendix C]{williams07-crossed-products} in the more general context of $\conto(X)$-algebras.
\begin{definition}
  \label{def:contX-algebras}
  Let $A$ be a unital \Cstar-algebra and $X$ a compact space.
  \begin{itemize}
  \item $A$ is called a $\cont(X)$-algebra if there is a unital embedding $\cont(X) \hra \mathcal{Z}(A)$.
  \item If $A$ is a $\cont(X)$-algebra and $x \in X$, then the fibre of $A$ at $x$ is defined as the quotient of $A$ by the ideal generated by $\{f \in \cont(X) \mid f(x) = 0\}$.  It is denoted $A_x$.  The image of $a \in A$ in $A_x$ is denote by $a_x$.
  \item $A$ is called a continuous $\cont(X)$-algebra, if it is a $\cont(X)$-algebra and the map $x \mapsto \|a_x\|$ is continuous for all $a \in A$.
  \end{itemize}
\end{definition}
If $A$ is a continuous $\cont(X)$-algebra we will also say that $A$ is a \Cstar-bundle over $X$.  This terminology is justified by \cite[Theorem C.26]{williams07-crossed-products} in connection with \cite[Proposition C.10 (b)]{williams07-crossed-products}, which explains the relation of $\cont(X)$-algebras with Fell's \Cstar-bundles \cite{fell61}.

In \cite[Proposition 3.3]{blanchardgogic15}, it was claimed that a unital $\cont(X)$-algebra $A$ with faithful conditional expectation $\rE: A \ra \cont(X)$ is a continuous $\cont(X)$-algebra.  This is not true as stated, but needs the additional assumption that every state $\rE_x: A_x \ra \CC$ defined by $\rE_x(a) = \rE(a)(x)$ is faithful.  Adopting this extra assumption, the proof from \cite{blanchardgogic15} applies.
\begin{proposition}[\mbox{See \cite[Proposition 3.3]{blanchardgogic15}}]
  \label{prop:blanchargogic}
  Let $A$ be a unital $\cont(X)$-algebra, with a conditional expectation $\rE: A \ra \cont(X)$ such that each $\rE_x: A_x \ra \CC$, $x \in X$ is faithful.  Then $A$ is a continuous $\cont(X)$-algebra.
\end{proposition}

\begin{notation}
  \label{not:glimm}
  If $A$ is a unital \Cstar-algebra, we denote the spectrum of $\cZ(A)$ by $\Glimm(A)$ and call it the Glimm space of $A$.  Note that $A$ is a $\cont(\Glimm(A))$-algebra and it is a \Cstar-bundle over $\Glimm(A)$ if there is a conditional expectation from $A$ onto its centre.
\end{notation}

\subsection{Twisted group C*-algebras}
\label{sec:twisted}

In this section we recall the notion of twisted group \Cstar-algebras, which is essential in the study of group \Cstar-algebras associated with nilpotent groups.  The following theorem is a special case of \cite[Theorem 1.2]{packerraeburn89-structure}.  However, since its proof for discrete groups becomes less technical, we provide it here for the convenience of the reader.
\begin{theorem}[Special case of {\cite[Theorem 1.2]{packerraeburn89-structure}}]
  \label{thm:packer-raeburn}
  Let $G$ be an amenable discrete group.  Write $Z = \cZ(G)$, $H = G/Z$ and denote by $\sigma \in \rZ^2(H, Z)$ an extension cocycle.  For $\chi \in \hat Z$ write $\sigma_\chi = \chi \circ \sigma \in \rZ^2(H, \rS^1)$.  Then $\Cstar(G)$ is a continuous $\cont(\hat Z)$-algebra, whose fibre at $\chi \in \hat Z$ is isomorphic with $\Cstar(H, \sigma_{\chi})$.  More precisely, for $\chi \in \hat Z$ we have a commutative diagram
  \begin{equation*}
    \xymatrix{
      \Cstar(G) \ar[r] \ar[dr]  & \Cstar(H, \sigma_\chi) \ar[d] \\
      & \Cstar(G)_\chi \eqstop
    }
  \end{equation*}
\end{theorem}
\begin{proof}
  Since $\Cstar(Z) \subset \cZ(\Cstar(G))$ and $\Cstar(Z) \cong \cont(\hat Z)$ via the Fourier transform, the \Cstar-algebra $\Cstar(G)$ is a $\cont(\hat Z)$-algebra.

  Recall that $\sigma \in \rZ^2(H, Z)$ denotes an extension cocycle.  This means that $G \cong Z \rtimes_\sigma H$ is the universal group containing a central copy of $Z$ and elements $\tilde h$ for $h \in H$ subject to the relation $\tilde h_1 \tilde h_2 = \sigma(h_1, h_2) \widetilde{h_1h_2}$ for all $h_1, h_2 \in H$.  It follows that $\Cstar(G)$ is the universal \Cstar-algebra generated by a central copy of $\Cstar(Z)$ and elements $u_h, h \in H$ such that $u_{h_1} u_{h_2} = u_{\sigma(h_1, h_2)} u_{h_1h_2}$ for all $h_1, h_2 \in H$.  The fibre of $\Cstar(G)$ at the element $\chi \in \hat Z$ is the quotient of $\Cstar(G)$ by the relations $u_z = \chi(z)$ for $z \in Z$.  It is hence isomorphic with the universal \Cstar-algebra generated by elements $v_h$ for $h \in H$ that satisfy $v_{h_1} v_{h_2} = \sigma_\chi(h_1, h_2) v_{h_1h_2}$ for all $h_1, h_2 \in H$.  This is precisely the twisted group \Cstar-algebra $\Cstar(H, \sigma_\chi)$.

  Since $G$ is amenable, its maximal and reduced group \Cstar-algebras agree so that the natural conditional expectation of reduced group \Cstar-algebras provides a faithful conditional expectation $\rE: \Cstar(G) \ra \Cstar(Z)$.  Since $\rE_\chi$ identifies with the natural trace $\Cstar(H, \sigma_\chi) \ra \CC$, it is faithful, so that Proposition \ref{prop:blanchargogic} implies that $\Cstar(G)$ is a continuous $\cont(\hat Z)$-algebra.  
\end{proof}

\noindent The next proposition is possibly known to experts and it was already used in \cite{kundbyraumthielwhite17}.  It provides a simple algebraic mean to detect which twisted group \Cstar-algebras are untwisted.
\begin{proposition}
  \label{prop:one-dimensional-projective-representation}
  Let $G$ be a discrete group and $\sigma \in \rZ^2(G, \rS^1)$.  If there is a one-dimensional $\sigma$-projective representation of $G$, then $\sigma$ is inner.  So the following statements are equivalent.
  \begin{itemize}
  \item $\sigma$ is inner.
  \item $\Cstar(G, \sigma)$ admits a character, that is a multiplicative positive functional.
  \item $\Cstar(G, \sigma) \cong \Cstar(G)$.
  \end{itemize}
\end{proposition}
\begin{proof}
  Let $\pi: G \ra \cU(1) = \rS^1$ be a one-dimensional $\sigma$-projective representation of $G$.  For all $g,h \in G$ we have $\pi(g)\pi(h) = \sigma(g,h) \pi(gh)$ and hence $\sigma(g,h) = \pi(g)\pi(h) \ol{\pi(gh)}$.  This proves that $\sigma$ is inner.    Since characters of $\Cstar(G, \sigma)$ are in 1-1 correspondence with one-dimensional $\sigma$-projective representations of $G$, we infer that $\Cstar(G, \sigma)$ has a character if and only only if $\sigma$ is inner.  The latter implies that $\Cstar(G, \sigma) \cong \Cstar(G)$, which in turn provides a character of $\Cstar(G, \sigma)$.
\end{proof}

\subsection{Group C*-algebras of torsion-free groups}
\label{sec:cstar-torsion-free-groups}

In this section we recall a characterisation of group C*-algebras of torsion-free nilpotent groups.  In fact, the following theorem holds for a much larger class of groups by Higson and Kasparov's \cite{higsonkasparov01} (see \cite[Chapter 6.3]{valette02-baum-connes} for detailed explanations).  Already in \cite[Proposition 2.5]{rosenberg83} and \cite[Corollary, p. 96]{kaniuthtaylor89} results appeared that apply to show that group \Cstar-algebras of torsion-free nilpotent groups admit no non-trivial projections.  An elementary proof, whose idea was communicated to us by Alain Valette, was provided in \cite[Theorem 8]{jipedersen90}.  For the convenience of the reader we provide the short argument here.
\begin{theorem}
  \label{thm:chacterising-tf}
  Let $G$ be a discrete group such that $\Cstar(G)$ has no non-trivial projections.  Then $G$ is torsion-free.  Conversely, if $G$ is torsion-free and nilpotent then $\Cstar(G)$ has no non-trivial projections.
\end{theorem}
\begin{proof} The first statement is clear, since for any group $G$ with a non-trivial torsion element $g \in G$, the formula $\frac{1}{\mathrm{ord(g)}} \sum_{n = 1}^{\mathrm{ord}(g)} u_g^n \in \CC G \subset \Cstar(G)$ defines a non-trivial projection.
 
 Let now $G$ be a torsion-free nilpotent group of nilpotency class $c$.  If $c = 0$, then $G$ is trivial and the conclusion is clear.  Assume that $c \geq 1$ and the statement is proven for $c - 1$.  Write $Z = \cZ(G)$ and let $\sigma \in \rZ^2(G/Z , Z)$ be an extension cocycle for $Z \hra G \thra G/Z$.  Theorem \ref{thm:packer-raeburn} says that $\Cstar(G)$ is a \Cstar-bundle over $\hat Z$ with fibre $\Cstar(G/Z, \chi \circ \sigma)$ at $\chi \in \hat Z$.  Let $p \in \Cstar(G)$ be a projection.  Denote by $p_\chi$ the image of $p$ in the fibre at $\chi$.  Denoting by $\veps$ the trivial character on $Z$, we find that $p_\veps \in \{0,1\}$ by the induction hypothesis.  Replacing $p$ by $1 - p$ if necessary, we may assume that $p_\veps = 0$.  Since $Z$ is torsion-free, $\hat Z$ is connected so that $\chi \mapsto \|p_\chi\|$ is a $\{0,1\}$-valued function on a connected space, which takes the value $0$.  Hence it is constantly $0$ and it follows that $p = 0$.  This finishes the proof of the theorem.
\end{proof}

\subsection{The primitive ideal space of nilpotent groups}
\label{sec:primitive-ideal-space}

In this section we recall results of Moore-Rosenberg \cite{moorerosenberg75} on the primitive ideal space of nilpotent groups.

\begin{definition}
  \label{def:primitive-ideal-space}
  Let $A$ a be a \Cstar-algebra.  The set of primitive ideals of $A$ is
  \begin{equation*}
    \Prim(A) = \{\ker \pi \mid \pi \text{ non-zero irreducible *-representation of } A\}
    \eqstop
  \end{equation*}
  This set is endowed with the unique topology whose closure operation satisfies
  \begin{equation*}
    \ol{X} = \{I \in \Prim(A) \mid I \supset \bigcap_{J \in X} J\}
    \qquad
    \text{for all } X \subset \Prim(A)
    \eqstop
  \end{equation*}
\end{definition}
It follows from the definition that a primitive ideal $I \unlhd A$ is a closed point in $\Prim(A)$ if and only if it is a maximal ideal.  Recall in this context that a topological space is called $\Tone$ if all its points are closed. Motivated by analogies with Lie groups, the following theorem is stated by Moore-Rosenberg only for finitely generated solvable groups $G$.  Its proof however, is valid for arbitrary amenable groups as the authors mention \cite[p.220]{moorerosenberg75}.

\begin{theorem}[\mbox{Moore-Rosenberg \cite[Theorem 5]{moorerosenberg75}}]
  \label{thm:tone-iff-virt-nilpotent}
  Let $G$ be a finitely generated group.
  \begin{itemize}
  \item If $G$ is virtually nilpotent, then $\Prim(G)$ is a \Tone-space.
  \item If $G$ is amenable and $\Prim(G)$ is a \Tone-space, then $G$ is virtually nilpotent.
  \end{itemize}
\end{theorem}

\subsection{Elliott's work on noncommutative tori}
\label{sec:elliott}

In the article \cite{elliott84}, Elliott undertook a detailed study of the K-theory of noncommutative tori, in particular calculating the trace paring with the $\rK_0$-group.  In view of Theorem \ref{thm:packer-raeburn}, these results will be essential to our investigations of group \Cstar-algebras of torsion-free finitely generated 2-step nilpotent groups.  In this section we recall the results of Elliott that will be used and we fix our notation.

\begin{proposition}[\mbox{See \cite[Lemma 2.3]{elliott84}}]
  \label{prop:elliott-unique-pairing}
  Let $A$ be a torsion-free abelian group and $\omega \in \Hom(A \wedge A, \rS^1)$.  There is a unique map $\rK_0(\tau): \rK_0(\Cstar(A, \omega)) \ra \RR$ induced by tracial states on $\Cstar(A, \omega)$.
\end{proposition}

\begin{theorem}[\mbox{See \cite[Theorem 2.2 and its proof]{elliott84}}]
  \label{thm:elliott-k-theory}
  Let $A$ be a torsion-free abelian group and $\gamma: [0,1] \ra \Hom(A \wedge A, \rS^1)$ a path with contractible image from $\omega_0$ to $\omega_1$.  Consider $\gamma$ as a $\cont([0,1])$-valued 2-cocycle on $A$ and denote by $\cont([0,1]) \rtimes_\gamma A$ the twisted crossed product.  Then the evaluation maps $\ev_0: \cont([0,1]) \rtimes_\gamma A \ra \Cstar(A, \omega_0)$ and $\ev_1: \cont([0,1]) \rtimes_\gamma A \ra \Cstar(A, \omega_1)$ induce isomorphisms in K-theory.  Further, the composition
  \begin{equation*}
    \pi_{\RR / \ZZ} \circ \rK_0(\tau) \circ \rK_0(\ev_1) \circ \rK_0(\ev_0)^{-1}:
    \rK_0(\Cstar(A, \omega_0)) \longrightarrow \RR / \ZZ
  \end{equation*}
  does not depend on the choice of $\gamma$.
\end{theorem}

\begin{notation}
  \label{not:trace}
  If $A$ is a torsion-free abelian group and $\omega \in \Hom(A \wedge A, \rS^1)$ is connected to the trivial form, we denote by
  \begin{equation*}
    \tau_\omega:\rK_0(\Cstar(A)) \longrightarrow \RR / \ZZ
  \end{equation*}
  the unique map arising form Theorem \ref{thm:elliott-k-theory}.
\end{notation}

\noindent Before stating the next theorem of Elliott, recall that K-theory of an abelian \Cstar-algebra admits a cup-product (see e.g. \cite{atiyah64}).
\begin{theorem}[\mbox{See \cite[Theorem 3.1]{elliott84}}]
  \label{thm:elliot-evaluation}
  Let $A$ be a torsion-free abelian group and $\omega \in \Hom(A \wedge A, \rS^1)$ be connected to the trivial form.  Consider $A$ as a subgroup of $\rK_1(\Cstar A)$ via the embedding $A \hra \cU(\Cstar A)$.  Then
  \begin{equation*}
    \tau_\omega([a_1]_{\rK_1} \cup [a_2]_{\rK_1}) = \omega(a_1 \wedge a_2)
    \qquad
    \text{for all } a_1, a_2 \in A
    \eqstop
  \end{equation*}  
\end{theorem}

\section{The centre of group C*-algebras of virtually nilpotent groups}
\label{sec:centre}

In this section we  determine the centre of the group \Cstar-algebra of a torsion-free virtually nilpotent group.  It turns out to be immediately related to the centre of the group itself.  Let $G$ be a torsion-free virtually nilpotent group, denote by $N \unlhd G$ some maximal finite index nilpotent normal subgroup and let $F = G/N$ be the quotient group.  Choosing a section $s:F \ra G$ for the quotient map, we  define the conjugation action $G \ni g \mapsto s(f)gs(f)^{-1}$, which is  up to inner automorphisms independent of the choice of $s$.  This provides a group homomorphism $F \ra \Out(N)$.  Since $\cZ(N) \leq N$ is a characteristic subgroup, we can compose the latter homomorphism with the restriction map $\Out(N) \ra \Out(\cZ(N)) = \Aut(\cZ(N))$ to obtain an action of $F$ on $\cZ(N)$.  Having fixed this notation we restate Theorem \ref{introthm:centre} from the introduction, which is the main theorem of this section.
\setcounter{introtheorem}{2}
\begin{introtheorem}
  If $G$ is a torsion-free virtually nilpotent group and $N \unlhd G$ is a maximal finite index nilpotent normal subgroup with quotient $G/N = F$, then $\cZ(\Cstar G) = \Cstar(\cZ(N))^F$.
\end{introtheorem}

\noindent The next lemma was already stated and proven in \cite[Corollary 2]{malcev49-homogeneous} for torsion-free finitely generated nilpotent groups.  We need the more general version dropping the assumption of finite generation and hence provide a different reference.
\begin{lemma}[\mbox{\cite[Corollary 1.3]{jennings55}}]
  \label{lem:torsion-free-quotients}
  Let $G$ be a torsion-free nilpotent group.  Then $G/\cZ(G)$ is torsion-free.
\end{lemma}

\noindent The next lemma describes the set of finite conjugacy classes of a torsion-free virtually nilpotent group.  Its conclusion resembles known results for torsion-free nilpotent groups \cite{mclain55}.
\begin{lemma}
  \label{lem:finite-conjugacy-classes}
  Let $G$ be a torsion-free virtually nilpotent group with maximal normal finite index nilpotent subgroup $N \unlhd G$.  Then every finite conjugacy class of $G$ lies in $\cZ(N)$.  So the FC-centre of $G$ equals $\cZ(N)$.
\end{lemma}
\begin{proof}
  Let us start by showing that the FC-centre of $G$ equals $\cZ_G(N)$.  If $g \in \cZ_G(N)$, then $g$ has a finite conjugacy class in $G$, since $N \leq G$ has finite index.  Vice versa, let $g \in G$ have a finite conjugacy class and we will show that $g \in \cZ_G(N)$.  Then $C = \{g h g^{-1} h^{-1} \mid h \in N\}$ is finite and we first show that $C = \{e\}$, or equivalently $g \in \cZ_G(N)$.  Let $\{e\} = Z_0 \leq \cZ(N) = Z_1 \leq Z_2 \leq \dotsm \leq Z_n = N$ be the upper central series of $N$ and assume that $C \subset Z_{k+1}$ holds for some $k \geq 0$.  Since $Z_{k+1}/Z_k \leq N/Z_k$ is central, we obtain for all $h_1, h_2 \in N$ that
  \begin{equation*}
    [gh_1g^{-1}h_1^{-1}]_{Z_k} [gh_2g^{-1}h_2^{-1}]_{Z_k}
    =
    [gh_1g^{-1}]_{Z_k} [gh_2g^{-1}h_2^{-1}]_{Z_k} [h_1^{-1}]_{Z_k}
    =
    [g(h_1h_2)g^{-1} (h_1h_2)^{-1}]_{Z_k}
  \end{equation*}
  showing that the image of $C$ in $N/Z_k$ is a finite subgroup.  Since $N$ is torsion-free nilpotent, Lemma \ref{lem:torsion-free-quotients} says that also $N/Z_k$ is torsion-free.  This shows that the image of $C$ in $N/Z_k$ is trivial and hence $C \subset Z_k$.  Because $k \geq 0$ was arbitrary, we conclude that $C = \{e\}$ or equivalently $g \in \cZ_G(N)$.

  We showed that $\cZ_G(N)$ is the FC-centre of $G$.  Since every torsion-free FC-group is abelian by Proposition \ref{prop:torsion-free-fc}, it follows that $\cZ_G(N)$ is abelian.  We denote $H = \langle \cZ_G(N), N \rangle \leq G$ and observe that $\cZ(H) = \cZ_G(N)$, since $\cZ_G(N)$ is abelian.  Because of the isomorphism $H/\cZ(H) \cong N/\cZ(N)$ we conclude that $H$ is nilpotent.  Further, $H$ is normal in $G$.  Since $N$ is a maximal normal finite index nilpotent subgroup of $G$, it follows that $H = N$ or equivalently $\cZ_G(N) = \cZ(N)$, showing that the FC-centre of $G$ equals $\cZ(N)$.
\end{proof}

\noindent We are now ready to prove the main theorem of this section.
\begin{proofseparate}[Proof of Theorem \ref{introthm:centre}]
By Lemma \ref{lem:finite-conjugacy-classes} the FC-centre of $G$ equals $\cZ(N)$.  So Proposition \ref{prop:centre-general-group-cstar-algebra} applies to show that $\cZ(\Cstar(G)) = \Cstar(\cZ(N))^G$.  Since the conjugation action of $G$ on $\cZ(N)$ factors through the finite group $F$, we obtain the identification $\Cstar(\cZ(N))^G = \Cstar(\cZ(N))^F$.
This finishes the proof of the theorem.  
\end{proofseparate}

\section{Characterising group C*-algebras of nilpotent groups}
\label{sec:characterising}

In this section we will prove Theorem \ref{introthm:nilpotent}, which is not only one of our main results, but will furthermore allow us to systematically invoke the structure of nilpotent groups when proving \Cstar-superrigidity results in Section~\ref{sec:superrigidity}.   Before proceeding to the proof of Theorem \ref{introthm:nilpotent}, we prepare two lemmas providing important ingredients.
\begin{lemma}
  \label{lem:nilpotent-vs-virtually-nilpotent}
  Let $G$ be a torsion-free finitely generated nilpotent group and $H$ be a virtually nilpotent group.  If $\Cstar(G) \cong \Cstar(H)$, then $H$ is torsion-free finitely generated and nilpotent.
\end{lemma}
\begin{proof}
  We prove the lemma by induction on the nilpotency class of $G$.  If the class of $G$ is at most $1$, then $G$ is abelian.  Since torsion-free abelian groups are \Cstar-superrigid by \cite[Theorem 1]{scheinberg74}, the isomorphism $\Cstar(G) \cong \Cstar(H)$ implies $G \cong H$.  In particular, $H$ is torsion-free finitely generated and nilpotent.  Let us assume that $G$ has class $c$ and that the conclusion of the lemma is known for all groups of nilpotency class at most $c-1$.  Note that by Theorem \ref{thm:chacterising-tf}, the group $H$ is torsion-free.  Let $N \unlhd H$ be a maximal finite index nilpotent normal subgroup.  By Theorem \ref{introthm:centre}, we have the following identifications of the centre of $\Cstar(G)$ and $\Cstar(H)$.
  \begin{align*}
    \cZ(\Cstar G) & = \Cstar(\cZ(G)) \\
    \cZ(\Cstar H) & = \Cstar(\cZ(N))^H
                    \eqstop
  \end{align*}
  Since $N \unlhd H$ has finite index, the action of $H$ on $\cZ(N)$ has finite orbits.  Further, $\Cstar(\cZ(N))^H \cong \Cstar(\cZ(G)) \cong \cont(\bT^n)$ for some $n \in \NN$ follows from the identification $\Cstar(H) \cong \Cstar(G)$.   So Theorem \ref{thm:quotient-implies-free} applies to the action $H \grpaction{} \cZ(N)$ by conjugation and we obtain the isomorphism $\cZ(N) \cong \ZZ^n$ and know that $H$ acts trivially on $\cZ(N)$, implying that $\cZ(N) \subset \cZ(H)$.  So $\cZ(H) \subset \cZ(N)$ follows from maximality of $N \unlhd H$ and in particular $\cZ(\Cstar H) = \Cstar(\cZ(N))$.  Consider the fibre of $\Cstar(G) \cong \Cstar(H)$ over the trivial character of $\cZ(N)$, which by Theorem \ref{thm:packer-raeburn} on the one hand is isomorphic with $\Cstar(H/\cZ(N))$  and on the other is a twisted group \Cstar-algebra of $G/\cZ(G)$.  Since $\Cstar(H/\cZ(N))$ admits a character, Proposition \ref{prop:one-dimensional-projective-representation} shows that $\Cstar(G/\cZ(G)) \cong \Cstar(H/\cZ(N))$.  The induction hypothesis applies to show that $H/\cZ(N)$ is torsion-free finitely generated and nilpotent.  Since $\cZ(N) \cong \ZZ^n$ was already shown, we see that $H$ is a central extension of $\ZZ^n$ by a torsion-free finitely generated nilpotent group.  So $H$ itself is torsion-free finitely generated and nilpotent.
\end{proof}

\begin{lemma}
  \label{lem:subquotients}
  Let $G$ and $H$ be two torsion-free nilpotent groups.  If $\Cstar(G) \cong \Cstar(H)$, then $G$ and $H$ have the same class and the subquotients of their upper central series are isomorphic.
\end{lemma}
\begin{proof}
  We prove the lemma by induction on the nilpotency class of $G$.  If $G$ is of nilpotency class at most $1$, then $G$ is abelian.  Since torsion-free abelian groups are \Cstar-superrigid by \cite[Theorem 1]{scheinberg74}, it follows that $G \cong H$.  Assume now that the lemma is proven for all groups of nilpotency class at most $c - 1$ and let $G$ be of nilpotency class $c$.  By Theorem \ref{introthm:centre}, we have $\Cstar(\cZ(G)) = \cZ(\Cstar(G)) \cong \cZ(\Cstar(H)) = \Cstar(\cZ(H))$.  In particular, $\cZ(G) \cong \cZ(H)$ by \Cstar-superrigidity of torsion-free abelian groups.  The fibre of $\Cstar(G)$ over the trivial character of $\cZ(G)$ is isomorphic with $\Cstar(G / \cZ(G))$, by Theorem \ref{thm:packer-raeburn}.  The same theorem combined with Proposition \ref{prop:one-dimensional-projective-representation} shows that this fibre is isomorphic with $\Cstar(H / \cZ(H))$.  So the induction hypothesis applies and shows that $G/ \cZ(G)$ and $H / \cZ(H)$ have the same nilpotency class $c-1$ and the subquotients of their upper central series agree.  It follows that $G$ and $H$ have nilpotency class $c$.  Moreover, we already showed that $\cZ(G) \cong \cZ(H)$ completing the proof that the subquotients of the upper central series of $G$ and $H$ are isomorphic. 
\end{proof}

\noindent Let us now proceed to the proof of this section's main theorem. 

\begin{proofseparate}[Proof of Theorem \ref{introthm:nilpotent}]
  We prove this theorem by induction on the nilpotency class $c$ of $G$.  If $c = 0$, then $G$ is trivial and the conclusion of the theorem is obvious.  If $c = 1$, then $G$ is a torsion-free abelian group and the conclusion follows from \Cstar-superrigidity of such groups.   Let us now assume that the theorem is proven for groups of nilpotency class at most $c - 1$ and we will show it for $G$ of nilpotency class $c$.  We first collect properties of $\Cstar(G)$ used in the subsequent arguments.
  \begin{itemize}
  \item $\Cstar(G)$ is amenable since $G$ is amenable.
  \item It admits no non-trivial projections, by Theorem \ref{thm:chacterising-tf}.
  \item Its centre is isomorphic with $\cont(\bT^n)$ for some $n \in \NN$, by Theorem \ref{introthm:centre}.
  \item Its primitive ideal space is $\Tone$, by Theorem \ref{thm:tone-iff-virt-nilpotent}.
  \end{itemize}
  From the isomorphism $\Cstarred(H) \cong \Cstar(G)$ we conclude that $H$ is amenable.  Further by Theorem \ref{thm:chacterising-tf},   $H$ is torsion-free.  We next show that $H$ is finitely generated.  Applying Proposition \ref{prop:centre-general-group-cstar-algebra}, we find that there is a normal abelian subgroup $A \leq H$ on which $H$ acts with finite orbits such that $\cZ(\Cstar(H)) = \Cstar(A)^H$.  Combined with the isomorphism $\cZ(\Cstar(H)) \cong \cZ(\Cstar G) \cong \cont(\bT^n)$ for some $n \in \NN$, Theorem \ref{thm:quotient-implies-free} applies to show that $H$ acts trivially on $A$.  It follows that $\cZ(H) = A$ and $\cZ(\Cstar(H)) = \Cstar(A)$.  Note that the fibre of $\Cstar(H)$ at the trivial character of $A$ is $\Cstar(H/A)$.  So $\Cstar(H/A)$ is isomorphic with some fibre of $\Cstar(G)$.  By Theorems \ref{introthm:centre} and \ref{thm:packer-raeburn}, it follows that $\Cstar(H/A)$ is isomorphic with a twisted group \Cstar-algebra of the quotient $G/\cZ(G)$.  Now Proposition \ref{prop:one-dimensional-projective-representation} implies that the twist is trivial and $\Cstar(H/A) \cong \Cstar(G/\cZ(G))$.  So the induction hypothesis applies to show that $H/A$ is a torsion-free finitely generated nilpotent group.  Because $A$ and $H/A$ are finitely generated,  $H$ is also finitely generated.  Since $H$ is finitely generated and has a $\Tone$ primitive ideal space, Theorem \ref{thm:tone-iff-virt-nilpotent} says that $H$ is virtually nilpotent.  Lemma \ref{lem:nilpotent-vs-virtually-nilpotent} applies to show that $H$ is actually torsion-free finitely generated and nilpotent.  Lemma \ref{lem:subquotients} says that the class of $H$ and the class of $G$ agree and that subquotients of their upper central series are isomorphic.
\end{proofseparate}

\section{C*-superrigidity of 2-step nilpotent groups}
\label{sec:superrigidity}

In this section we prove Theorem \ref{introthm:superrigid} showing that all torsion-free finitely generated 2-step nilpotent groups are \Cstar-superrigid.  Theorem \ref{introthm:superrigid} is derived from the technical Theorem \ref{thm:recovering-triple}, whose notation we now introduce.

  Let $G$ be a torsion-free finitely generated 2-step nilpotent group with centre $Z.$  Theorem \ref{introthm:centre} provides a natural isomorphism $\hat \vphi_Z: \hat Z \ra \Glimm(\Cstar G)$ induced by the Fourier transform.  We denote the induced isomorphism of fundamental groups by $\vphi_Z = \pi_1(\hat \vphi_Z): Z \ra \pi_1(\Glimm (\Cstar G))$.  Further, if $\veps \in \Glimm(\Cstar G)$ is an element such that the fibre $\Cstar(G)_\veps$ admits a character, we obtain a natural isomorphism $\Cstar(G)_\veps \cong \Cstar(G/Z)$ from Theorem \ref{thm:packer-raeburn} and Proposition \ref{prop:one-dimensional-projective-representation}.  In particular, the natural map $G/Z \ra \rK_1(\Cstar(G)_\veps)$ is an isomorphism onto its image, which equals $\im \bigl (\cU(\Cstar G) \ra \rK_1(\Cstar G)_\veps \bigr)$.  We denote this isomorphism by $\vphi_{\mathrm{ab}}:  G/Z \ra \im (\cU(\Cstar G) \ra \rK_1(\Cstar G)_\veps)$.

  Recall the notation and Elliott's results on K-theory of noncommutative tori from Section \ref{sec:elliott}. In particular,for $\chi \in \Glimm(\Cstar G)$ the isomorphism $\Cstar(G)_\chi \cong \Cstar(G/Z, \chi \circ \sigma)$ from Theorem \ref{thm:packer-raeburn} provides an isomorphism in K-theory, so that Elliott's work gives rise to an isomorphism $\rK_*(\Cstar(G)_\veps) \ra \rK_*(\Cstar(G)_\chi)$, whose restriction defines a unique map $\rK_0(\Cstar(G)_\veps) \ra \rK_0(\Cstar(G)_\chi) \ra \RR/\ZZ$ as described in Theorem \ref{thm:elliott-k-theory}.  We denote this map by $\tau_\chi$.

 We also make use of the cup-product in K-theory of the abelian \Cstar-algebra $\Cstar(G)_\veps$, for which we refer to \cite{atiyah64}.
 
\begin{theorem}
  \label{thm:recovering-triple}
  Let $G$ be a torsion-free finitely generated 2-step nilpotent group.   Let $\veps \in \Glimm(\Cstar G)$ such that $\Cstar(G)_\veps$ admits a character.  Put
  \begin{align*}
    A & = \pi_1(\Glimm(\Cstar G)) \eqcomma \\
    B & = \im \bigl (\cU(\Cstar G) \ra \rK_1(\Cstar G)_\veps \bigr )
        \eqstop
  \end{align*}
  Choose some isomorphism $A \cong \ZZ^m$ inducing a bilinear form on $A$ and define an $A$-valued skew-symmetric form $\omega \in \Hom(B \wedge B , A)$ by the formula
  \begin{equation*}
    \langle \omega(b_1 \wedge b_2), a \rangle = \pi_1(f_{b_1 \wedge b_2})(a)
    \qquad
    \forall a \in A
  \end{equation*}
  where
  \begin{equation*}
    f_{b_1 \wedge b_2}: \Glimm(\Cstar G) \ra \RR/\ZZ:
    f_{b_1 \wedge b_2}(\chi) = \tau_\chi(b_1 \cup b_2)
    \eqstop
  \end{equation*}
    Then $(A, B, \omega) \sim (\cZ(G), G/\cZ(G), \omega_G)$ in the sense of Corollary \ref{cor:extension-data-skew-symmetric}.
\end{theorem}
\begin{proof}
  As in the introduction of this section, we will write $Z = \cZ(G)$.  We start by showing that 
  \begin{equation}
    \label{eq:formula}
    f_{b_1 \wedge b_2}
    =
    \ev_{\omega_G \circ (\vphi_{\mathrm{ab}}^{-1} \wedge \vphi_{\mathrm{ab}}^{-1})(b_1 \wedge b_2) }\circ \hat \vphi_Z^{-1}
    \eqstop
  \end{equation}
  Let $b_1, b_2 \in B$ and write $g_i = \vphi_{\mathrm{ab}}^{-1}(b_i) \in G/Z$ for $i \in \{1, 2\}$. Let $\chi \in \Glimm(\Cstar G)$, consider $\hat \vphi_Z^{-1}(\chi) \in \hat Z$ and note that $\hat \vphi_Z^{-1}(\chi) \circ \omega_G \in \Hom(G/Z \wedge G/Z, \rS^1)$.  So Theorem \ref{thm:elliot-evaluation} applies and says that
  \begin{equation*}
    \tau_{\hat \vphi_Z^{-1}(\chi)\circ \omega_G}([g_1]_{\rK_1} \cup [g_2]_{\rK_1})
    =
    \hat \vphi_Z^{-1}(\chi) \circ \omega_G (g_1 \wedge g_2)
    \eqstop
  \end{equation*}
  The left-hand side can be rewritten as
  \begin{equation*}
    \tau_{\hat \vphi_Z^{-1}(\chi)\circ \omega_G}([g_1]_{\rK_1} \cup [g_2]_{\rK_1})
    =
    \tau_\chi(b_1 \cup b_2)
  \end{equation*}
  thanks to the identification $\Cstar(G)_\chi \cong \Cstar(G/Z, \chi \circ \omega_G)$ from Theorem \ref{thm:packer-raeburn} and the definition of $g_1, g_2$.  We obtain
      \begin{equation*}
      f_{b_1 \wedge b_2}(\chi)
      =
      \hat \vphi_Z^{-1}(\chi)
      \Bigl (
      \omega_G \circ (\vphi_{\mathrm{ab}}^{-1} \wedge \vphi_{\mathrm{ab}}^{-1})(b_1 \wedge b_2)
      \Bigr )
      \qquad
      \text{for all } \chi \in \Glimm(\Cstar G)
      \eqcomma
    \end{equation*}
    which is equivalent to formula (\ref{eq:formula}).

    We consider $Z$ equipped with the bilinear form induced from $\vphi_Z: Z \ra A$.  Let $h: \rS^1 \ra \Glimm(\Cstar G)$ be a loop.  Then 
  \begin{align*}
    \langle \omega \circ (\vphi_{\mathrm{ab}} \wedge \vphi_{\mathrm{ab}})(g_1 \wedge g_2), [h] \rangle
    & =
      \pi_1(f_{b_1 \wedge b_2})([h])
      \tag{definition of $\omega$} \\
    & =
      \pi_1(\ev_{\omega_G (g_1 \wedge g_2)} \circ \hat \vphi_Z^{-1})([h])
      \tag{formula (\ref{eq:formula})} \\
    & =
      \pi_1(\ev_{\omega_G (g_1 \wedge g_2)})([\hat \vphi_Z^{-1} \circ h])
      \tag{elementary property of $\pi_1$} \\
    & =
      \langle \omega_G(g_1 \wedge g_2) , [\hat \vphi_Z^{-1} \circ h] \rangle
      \tag{Proposition \ref{prop:pairing-zn}} \\
    & =
      \langle \omega_G(g_1 \wedge g_2) , \vphi_Z^{-1} ([h]) \rangle
      \tag{definition of $\vphi_Z$} \\
    & =
      \langle \vphi_Z \circ \omega_G(g_1 \wedge g_2) , [h] \rangle \eqstop
      \tag{definition of the bilinear form on $Z$}
  \end{align*}
  It follows that $\omega  \circ (\vphi_{\mathrm{ab}} \wedge \vphi_{\mathrm{ab}}) = \vphi_Z \circ \omega_G$.  This finishes the proof of the theorem.
\end{proof}

\noindent We are now ready to prove the main result of this article.
\begin{proofseparate}[Proof of Theorem \ref{introthm:superrigid}]
  Let $G$ be a torsion-free finitely generated 2-step nilpotent group and $H$ some group such that $\Cstarred(G) \cong \Cstarred(H)$.  Theorem \ref{introthm:nilpotent} says that $H$ is also a torsion-free finitely generated nilpotent group.  Denote by $(A, B, \omega)$ the triple constructed from $\Cstar(G)$ by means of Theorem \ref{thm:recovering-triple}.  Then $(A, B, \omega)$ is equivalent to $(\cZ(G), G/\cZ(G), \omega_G)$ and to $(\cZ(H), H/cZ(H), \omega_H)$ by Theorem \ref{thm:recovering-triple} and the isomorphism $\Cstar(G) \cong \Cstar(H)$.  So Corollary \ref{cor:extension-data-skew-symmetric} implies that $G \cong H$, finishing the proof of the theorem.
\end{proofseparate}




\vspace{2em}
{\small \parbox[t]{200pt}
  {
    Sven Raum \\
    Department of Mathematics \\
    Stockholm University \\
    SE-106 91 Stockholm \\
    Sweden \\
    {\footnotesize raum@math.su.se}
  }
  \hspace{15pt}
  \parbox[t]{200pt}
  {
    Caleb Eckhardt \\
    Department of Mathematics, Miami University \\
    Oxford, OH 45056, USA \\
    {\footnotesize eckharc@miamioh.edu}
  }
}

\end{document}